\documentclass[12pt]{amsart}

\usepackage{amsfonts,amssymb,amsmath,amsthm}
\usepackage{url}
\usepackage{enumerate}

\newtheorem{theorem}{Theorem}[section]
\newtheorem{lemma}[theorem]{Lemma}
\newtheorem{proposition}[theorem]{Proposition}
\newtheorem{corollary}[theorem]{Corollary}
\theoremstyle{definition}
\newtheorem{definition}[theorem]{Definition}

\numberwithin{equation}{section}
\newtheorem{example}[theorem]{Example}

\renewcommand{\le}{\leqslant}
\renewcommand{\ge}{\geqslant}
\renewcommand{\mid}{\::\:}

\def\rk{{\rm rank}\,}

\def\hilb{\mathcal H}

\def\bbC{\mathbb C}

\def\bbN{\mathbb N}

\def\bbR{\mathbb R}

\def\bbT{\mathbb T}

\def\cB{\mathcal B}
\def\cC{\mathcal C}

\def\cG{\mathcal G}

\def\cM{\mathcal M}
\def\cN{\mathcal N}

\def\cS{\mathcal S}

\def\diag{\mathrm{diag}}
\def\span{\mathrm{span}}

\author[A.~Jafarian]{Ali Jafarian}
\address{University of New Haven, 300 Boston Post Rd., West Haven, CT 06516, USA
}
\email{ajafarian@newhaven.edu}

\author[A.I.~Popov]{Alexey I. Popov${}^1$}
\address{Department of Pure Mathematics, University of Waterloo, 200 University Avenue West, Waterloo, ON, N2L\,3G1, Canada}
\email{a4popov@uwaterloo.ca}

\author[M.~Radjabalipour]{Mehdi Radjabalipour${}^2$}
\address{Department of Pure Mathematics, University of Waterloo, 200 University Avenue West, Waterloo, ON, N2L\,3G1, Canada (On sabbatical from the Iranian Academy of Sciences, Tehran, Iran)}
\email{radjabalipour@ias.ac.ir}

\author[H.~Radjavi]{Heydar Radjavi${}^1$}
\address{Department of Pure Mathematics, University of Waterloo, 200 University Avenue West, Waterloo, ON, N2L\,3G1, Canada}
\email{hradjavi@uwaterloo.ca}

\thanks{${}^1$ Research supported in part by NSERC (Canada)}
\thanks{${}^2$ Research supported in part by the Iranian National Science Foundation}


\begin{document}
\baselineskip 18pt

\title[Reducibility of operator semigroups]
{Commutators of small rank and reducibility of operator semigroups}

\begin{abstract}
It is easy to see that if $\cG$ is a non-abelian group of unitary matrices, then for no members $A$ and $B$ of $\cG$ can the rank of $AB-BA$ be one. We examine the consequences of the assumption that this rank is at most two for a general semigroup $\cS$ of linear operators. Our conclusion is that under obviously necessary, but trivial, size conditions, $\cS$ is reducible. In the case of a unitary group satisfying the hypothesis, we show that it is contained in the direct sum $\cG_1\oplus\cG_2$ where $\cG_1$ is at most $3\times 3$ and $\cG_2$ is abelian.
\end{abstract}

\maketitle
\section{Introduction}\label{introduction} 

It is easy to see that if $\cG$ is a non-abelian group of unitary matrices, then for no members $A$ and $B$ of $\cG$ can the rank of $AB-BA$ be one. Indeed, suppose that $A,B\in\cG$ be such that $AB\ne BA$. Then $ABA^{-1}B^{-1}-I=(AB-BA)A^{-1}B^{-1}$. Since $ABA^{-1}B^{-1}$ is a member of $\cG$, it is a unitary matrix, hence it is diagonalizable via a unitary similarity. If the rank of $AB-BA$ were equal to one, exactly one diagonal entry of $ABA^{-1}B^{-1}$ would be different from one, so that $\det(ABA^{-1}B^{-1})$ would be different from one, which is, clearly, a contradiction. In particular, this shows that the condition $\rk(AB-BA)\le 1$ for all $A,B$ in a unitary group $\cG$ implies that $\cG$ is abelian.

For semigroups of matrices and, more generally, linear operators on Banach spaces, the corresponding problem is more difficult. The following result was obtained in \cite[Corollary~2]{RR97}.

\begin{theorem}[\cite{RR97}]\label{first-triang-result}
Let $\cS$ be a semigroup of Schatten $p$-class operators on a Hilbert space. If $\rk(AB-BA)\le 1$ for all $A,B\in\cS$, then $\cS$ is triangularizable. 
\end{theorem}

This was generalized to compact operators on arbitrary Banach spaces in \cite[Theorem 9.2.10]{RR00}. For non-compact operators, this question was studied in a series of papers. In \cite[Lemma~5]{Cig}, the authors showed that the same conclusion holds for semigroups of algebraic operators, and in \cite{Drn00}, it was shown that every non-commutative doubly generated semigroup $\cS$ with the condition that $\rk(AB-BA)\le 1$ for all $A,B\in\cS$ has a hyperinvariant subspace. Finally, it was generalized to arbitrary operators on Banach spaces in~\cite{Drn09} as follows:

\begin{theorem}[\cite{Drn09}]\label{rk-one-general}
Let $X$ be a Banach space of dimension at least two. Let $\cS$ be a non-commutative semigroup of operators on~$X$. If $\rk(AB-BA)\le 1$ for all $A,B\in\cS$ then $\cS$ is reducible. 
\end{theorem}

It is natural to try to replace the rank-one condition in the above statements with the condition $\rk(AB-BA)\le r$, where $r\in\bbN$ is fixed. The following quick example shows that one cannot expect the same answer as in Theorem~\ref{rk-one-general} even for semigroups of finite-rank operators.

\begin{example}\label{E_ij}
Let $\hilb$ be a finite- or infinite-dimensional Hilbert space. For all $i,j=1,2,\dots$, denote the $i,j$-matrix unit by $E_{ij}$ . That is, for a fixed orthonormal basis $(e_i)$, we have $E_{ij}(e_k)=\delta_{jk}e_i$. The semigroup
$$
\cS=\{E_{ij}\mid i,j\in\bbN\}\cup\{0\}
$$
is an irreducible semigroup of operators of rank $\le 1$ such that $\rk(AB-BA)\le 2$ for all $A,B\in\cS$.
\end{example}

In the present paper, we obtain results regarding the following question: when does the assumption $\rk(AB-BA)\le 2$ for all operators $A$ and $B$ in a semigroup $\cS$ imply reducibility of~$\cS$? Our main argument uses special unitary groups whose structure is also of some independent interest and is a subject of study of the last section of this paper.

Throughout the paper, the linear space $\bbC^n$ is considered as a Hilbert space with the standard inner product $\langle\cdot,\cdot\rangle$. In the case of infinite dimensional spaces, the term \emph{operator} is reserved for the bounded linear operators. The set of operators on a Banach space $X$ is denoted by $\cB(X)$. The term \emph{invariant subspace} means a non-trivial invariant subspace. A \emph{semigroup} is a set $\cS$ of operators on $X$ such that $AB\in\cS$ for all $A,B\in\cS$. A semigroup $\cS\subseteq\cB(X)$ is \emph{reducible} if it admits an invariant subspace, and it is \emph{triangularizable} if there exists a chain $\cC$ that is maximal as a chain of  subspaces of $X$ and that has the property that every member of $\cC$ is $\cS$-invariant (see \cite[Definition 7.1.1]{RR00}). A semigroup $\cS\subseteq\cB(X)$ is \emph{irreducible} if it is not reducible. The symbol $\diag\{\alpha_1,\alpha_2,\dots,\alpha_n\}$ denotes the $n\times n$ diagonal matrix with $\alpha_1,\alpha_2,\dots,\alpha_n$ on the diagonal. The symbol $\mathrm{nul}(A)$ denotes the dimension of $\ker A$. Finally, we will write $A\equiv B$ if the matrices $A$ and $B$ are unitarily similar.


\section{Reducibility of Semigroups}\label{main-section} 

We will start by investigating the structure of certain very special groups of unitaries. 

\begin{definition}\label{G(p,q,A)}
Let $p$ and $q$ be two prime numbers. The symbol $\cG(p,q,A)$ will denote the group of unitaries generated by the $p\times p$ matrices
$$
S=
\begin{bmatrix}
0 & 0 & \dots & 0 & 1\\
1 & 0 & \dots & 0 & 0\\
0 & 1 & \dots & 0 & 0\\
\vdots &  &  &  & \\
0 & 0 & \dots & 1 & 0
\end{bmatrix}
\quad
\mbox{and}
\quad
A=\begin{bmatrix}
\omega_1 & 0 & \dots & 0 & 0\\
0 & \omega_2 & \dots & 0 & 0\\
\vdots &  &  & & \\
0 & 0 & \dots &  \omega_{p-1} & 0\\
0 & 0 & \dots & 0 & \omega_p
\end{bmatrix},
$$
where $A$ is not a scalar multiple of the identity and $\omega_i^q=1$ for all $i=1,2,\dots,p$.
\end{definition}

Our interest in these groups stems from the fact that if $\cG$ is a minimal non-abelian group of matrices, then $\cG$ admits a subgroup $\cG_0$ whose restriction to a $\cG_0$-invariant subspace is closely related to a group of the form $\cG(p,q,A)$ (see \cite[Lemma~4.2.9]{RR00}). 

\begin{proposition}\label{p-q-restrictions}
Let $p,q$ be two prime numbers and $A$ be a $p\times p$ matrix as in Definition~\ref{G(p,q,A)}. If $\rk(XY-YX)\le 2$ for all $X,Y\in\cG(p,q,A)$, with the equality achieved on some members of it, then either
\begin{enumerate}
\item $p=2$ or
\item $p=3$ and $q=2$.
\end{enumerate}
\end{proposition}
\begin{proof}
Denote the group $\cG(p,q,A)$ by $\cG$, for simplicity of notations. It is not hard to see that every member of $\cG$ can be written in the form $DS^k$ where $D$ is a diagonal matrix whose diagonal entries are $q$-roots of unity, $S$ is the cyclic permutation as in Definition~\ref{G(p,q,A)}, and $0\le k<p$. Moreover, if $X_1=D_1S^{k_1}$ and $X_2=D_2S^{k_2}$, then $X_1X_2=D_3S^{k_1+k_2}$, for some diagonal matrix $D_3$.

Let $X$ and $Y$ be arbitrary members of $\cG$. It follows from the above observation that $XYX^{-1}Y^{-1}$ is a diagonal matrix. It is clear that if $\rk(XY-YX)=2$, then exactly two eigenvalues of $XYX^{-1}Y^{-1}$ are not equal to one. Since $\det(XYX^{-1}Y^{-1})=1$, we conclude that $XYX^{-1}Y^{-1}$ is of the form $\diag(1,\dots,1,\omega,1,\dots,1,\bar\omega,1,\dots,1)$, where $\omega\ne 1$ and $\omega^q=1$, and each of the series of ones between $\omega$ and $\bar\omega$ could be absent.

Observe that if $D=\diag(d_1,\dots,d_{p-1},d_p)$, then $SDS^{-1}=\diag(d_p,d_1,\dots,d_{p-1})$. It follows that $\cG$ has a member of the form
$$
A_0=\diag(\omega,1,\dots,1,\bar\omega,1\dots,1),
$$
where $\omega\ne 1$, $\omega^q=1$, and the series of ones between $\omega$ and $\bar\omega$ is shorter than the series of ones following $\bar\omega$.

Suppose that $p>3$, so that $p\ge 5$. If the series of ones between $\omega$ and $\bar\omega$ is not absent, then consider $B=SA_0^{-1}S^{-1}$. It follows that 
$$
A_0B=A_0SA_0^{-1}S^{-1}=\diag(\omega,\bar\omega,1\dots,1,\bar\omega,\omega,1\dots,1),
$$
so that $\rk(A_0S-SA_0)=\rk(A_0SA_0^{-1}S^{-1}-I)=4$, contrary to the assumptions. So, the series of ones between $\omega$ and $\bar\omega$ must be absent, and 
$$
A_0=\diag(\omega,\bar\omega,1\dots,1).
$$
However, in this case we may consider $C=S^2A_0^{-1}S^{-2}$. We get
$$
A_0C=A_0S^2A_0^{-1}S^{-2}=\diag(\omega,\bar\omega,\bar\omega,\omega,1\dots,1),
$$
so that $\rk(A_0S^2-S^2A_0)=4$.

This shows that either $p=2$ or $p=3$. Suppose that $p=3$. We claim that, necessarily, $q=2$. Assume that $q>2$. Then, by the same argument as above,
$$
A_0=\diag(\omega,\bar\omega,1)\in\cG,
$$
where $\omega\ne 1$ and $\omega^q=1$. Clearly, $SA_0^{-1}S^{-1}=\diag(1,\bar\omega,\omega)$, so that 
$$
A_0SA_0^{-1}S^{-1}=\diag(\omega,\bar\omega^2,\omega).
$$
If $q>2$, then all the diagonal entries of this matrix are different from~$1$, so that $\rk(A_0S-SA_0)=3$, a contradiction.
\end{proof}

The next proposition records certain observations about the groups $\cG$ satisfying $\rk(XY-YX)\le 2$ for all $X,Y\in\cG$. We will need the following notation.

\begin{definition}\label{G(M)}
Let $\cS$ be a set of $n\times n$ matrices and $\cM$ be a linear subspace of $\bbC^n$. Then we put
$$
\cS(\cM)=\{T\in\cS\mid T\cM\subseteq\cM\}.
$$
\end{definition}

\begin{proposition}\label{t.HZ}
Let $\mathcal G$ be a non-abelian group of unitary $n\times n$ matrices, and assume ${\rm rank}(AB-BA)\leq 2$ for all $A,B \in \mathcal G$. If $\mathcal M$ is a linear subspace of $\bbC^n$, then $\mathcal G(\mathcal M)=\mathcal G(\mathcal M^\perp)$ is a subgroup of $\cG$ and at least one of the unitary groups $\mathcal G(\mathcal M)|_{\mathcal M}$ or $\mathcal G(\mathcal M)|_{\mathcal M^\perp}$ is abelian.
\end{proposition}

\begin{proof} If $n\le 2$, then the conclusions of the proposition are evident. Therefore, we will assume in the proof that $n\ge 3$.

Let $A,B\in \mathcal G$. Since $ABA^{-1}B^{-1}$ is a unitary and $\rk(ABA^{-1}B^{-1}-I)={\rm rank}(AB-BA)\leq 2$, it follows  from the first paragraph of the introduction that $\rk(AB-BA)$ is $0$ or $2$, and hence $ABA^{-1}B^{-1}\equiv {\rm diag}(\omega,\omega', 1,1,\cdots,1)\neq I$ for some $\omega\neq 1\ne\omega'$. Also, since $1=\det (ABA^{-1}B^{-1})=\omega\omega'$, it follows  that  $\omega'=\bar\omega$.  

Next, assume $\mathcal M$ is a linear subspace of $\bbC^n$. Clearly, $\mathcal G(\mathcal M)=\mathcal G(\mathcal M^\perp)$. Assume, if possible, that both $\mathcal G(\mathcal M)|_{\mathcal M}$ and $\mathcal G(\mathcal M)|_{\mathcal M^\perp}$  are non-abelian. For $i=1,2$,  choose $A_i=C_i\oplus D_i\in \mathcal G(\mathcal M)$  decomposed according to $\bbC^n=\mathcal M\oplus\mathcal M^\perp$,  such that $C_1C_2\neq C_2C_1$. Notice that the condition $D_1D_2\ne D_2D_1$ would imply $\rk(A_1A_2-A_2A_1)=4$, hence $D_1D_2=D_2D_1$. Assume, if possible, that $D_1$ is  not in the centre of $\mathcal G(\mathcal M)|_{\mathcal M^\perp}$. In this case, choose $A_3=C_3\oplus D_3\in\mathcal G(\mathcal M)$ such that $D_1D_3\neq D_3D_1$ and, consequently, $C_3C_1=C_1C_3$.   Then  $C_1(C_2C_3)\neq C_2C_1C_3=(C_2C_3)C_1$ and $D_1(D_2D_3)=D_2D_1D_3\neq (D_2D_3)D_1$, which is a contradiction. Thus,  $D_1$ and, by symmetry, $D_2$ belong to the centre of $\mathcal G(\mathcal M)|_{\mathcal M^\perp}$.   Now, since  $\mathcal G(\mathcal M)|_{\mathcal M^\perp}$ is not abelian, there exist $A_3=C_3\oplus D_3$ and $A_4=C_4\oplus D_4$ in $\mathcal G(\mathcal M)$ such that $C_3C_4=C_4C_3$ and $D_3D_4\neq D_4D_3$.  Another symmetrical argument reveals that $C_3,C_4$ belong to the centre of $\mathcal G(\mathcal M)|_{\mathcal M}$. Then $(C_1C_3)(C_2C_4)=C_1C_2C_3C_4\neq  C_2C_1C_3C_4=(C_2C_4)(C_1C_3)$ and $(D_1D_3)(D_2D_4)=D_2D_3D_4D_1\neq D_2D_4D_3D_1=(D_2D_4)(D_1D_3)$ and, hence, rank$[(A_1A_3)(A_2A_4)-(A_2A_4)(A_1A_3)] =4$; a contradiction. 
\end{proof}

Before we state our main theorem, we need two lemmas.

\begin{lemma}\label{t.G0p3q2} 
Let $\mathcal G$ be a non-abelian unitary group on $\bbC^n$ and $\cN\subseteq\bbC^n$ be a 3-dimensional subspace. Assume that $\cG$ has a subgroup $\cG_0$ such that $\cN$ is $\cG_0$-invariant and, in some basis $\{e_1,e_2,e_3\}$ of $\cN$, $\cG_0|_{\cN}=\cG(3,q,A)$, where $q$ is a prime number and $A$ a diagonal matrix as in Definition~\ref{G(p,q,A)}. If $\rk(XY-YX)\le 2$ for all $X,Y\in\cG$, then $\cN$ is $\cG$-invariant. Moreover, $\mathcal G|_{\cN}$ is irreducible and $\mathcal G|_{\cN^\perp}$ is abelian.
\end{lemma}

\begin{proof} 
Since $G$ is non-abelian, in view of the observation made at the beginning of the introduction, $\rk(XY-YX) = 2$, for some $X$, $Y\in\cG$. By Proposition~\ref{p-q-restrictions}, $q$ must be equal to~$2$. Considering matrices of the form $XYX^{-1}Y^{-1}$, as in the proof of Proposition~\ref{p-q-restrictions}, we conclude that $\cG(3,2,A)$ admits a diagonal matrix $B$ with eigenvalues $\{1,-1,-1\}$. Considering $SBS^{-1}$ and $S^2BS^{-2}$, where $S$ is the cyclic permutation as in Definition~\ref{G(p,q,A)}, we conclude that the matrices
$$
\diag(1,-1,-1),\quad\diag(-1,1,-1),\quad\mbox{and}\quad\diag(-1,-1,1)
$$
all belong to $\cG(3,2,A)$. 

Pick an arbitrary $Z\in\cG$ and assume that $\cN$ is not $Z$-invariant. Fix a matrix $\widetilde S\in\cG$ such that $\widetilde S|_\cN=S$. Choose a basis $\{e_4,e_5,\cdots,e_n\}$ for $\cN^\perp$ consisting of eigenvectors of $\widetilde S$. Since $\cN$ (and, hence, $\cN^\perp$) is not $Z$-invariant, there exist $i\leq 3$ and $j\geq 4$ such that $\langle Ze_j,e_i\rangle\neq 0$. Due to the cyclic nature of the conditions of the theorem with respect to the ordered triple $(e_1,e_2,e_3)$, we may and shall assume without loss of generality that $i=1$.  Let $\mathcal M$ be the $2$-dimensional subspace of $\bbC^n$ spanned by $\{e_1,e_2\}$ and write 
$$
Z=\left[\begin{array}{cc}Z_{11}& Z_{12}\\ Z_{21} &Z_{22}\end{array}\right]~{\rm with~respect~ to}~\bbC^n=\mathcal M\oplus\mathcal M^\perp.
$$
The matrix $\diag(-1,-1,1)\in\cG(3,2,A)$ can be obtained as $CSC^{-1}S^{-1}$ where $C=\diag(1,-1,-1)\in\cG(3,2,A)$. This shows that if $\widetilde C\in\cG$ is such that $\widetilde C|_{\cN}=C$, then the matrix
$$
T=\widetilde C\widetilde S\widetilde C^{-1}\widetilde S^{-1}=\diag(-1,-1,1,1,\dots,1)\in\cG.
$$
With respect to $\bbC^n=\cM\oplus\cM^\perp$, this matrix has the form
$$
T=
\left[\begin{array}{cc}
-I & 0\\ 
0 & I
\end{array}\right].
$$
Then 
$$
TZ-ZT=
\left[\begin{array}{cc}
0 & 2Z_{12}\\ 
-2Z_{21} & 0
\end{array}\right].
$$
Since $\rk(TZ-ZT)\le 2$, we conclude that $\rk(Z_{12})=\rk(Z_{21})=1$, for none of $Z_{12}$ and $Z_{21}$ are zero. It follows that $Z_{12}e_k$ $(k=3,4,\cdots,n)$ are multiples of $Z_{12}e_j$. Replacing $Z$ by $Z\widetilde S$ changes the first column $Z_{12}e_3$ of $Z_{12}$ to $Z_{11}e_1$ and its $(j-2)^{\rm nd}$ column $Z_{12}e_j$ to $\lambda_jZ_{12}e_j$, where $\lambda_j$ is the eigenvalue of $\widetilde S$ corresponding to~$e_j$. Thus, again, $Z_{11}e_2$ is a multiple of $Z_{12}e_j$. Another replacement of $Z$ by $Z\widetilde S^2$ reveals that the first two rows of $Z$ are linearly dependent; a contradiction. This shows that $\cN$ is $\cG$-invariant.

Finally, the irreducibility of $\mathcal G|_{\mathcal N}$ follows from the fact that $\cG(p,q,A)$ is irreducible (see, e.g., \cite[Lemma~4.2.8]{RR00}), and the commutativity of $\mathcal G|_{\mathcal N^\perp}$ was established in Proposition~\ref{t.HZ}.
\end{proof}

\begin{lemma}\label{shifted-inv-subspaces}
Let $\cS$ be a semigroup of $n\times n$ matrices and $\cN$ be a subspace of $\bbC^n$ such that, with respect to the decomposition $\bbC^n=\cN\oplus\cN^\perp$, the representation of every member $Z\in\cS$
$$
Z=\begin{bmatrix}
Z_{11} & Z_{12}\\
Z_{21} & Z_{22}
\end{bmatrix}
$$
has the property that $\rk(Z_{21})\le 1$. Then each $Z\in\cS$ admits an invariant subspace $\cN_Z$ such that either $\cN_Z\subseteq\cN$, in which case $\dim(\cN/\cN_Z)\le 1$, or $\cN\subseteq\cN_Z$, in which case $\dim(\cN_Z/\cN)\le 1$ and $\cN_Z=\span\{\cN,Z\cN\}$.
\end{lemma}

\begin{proof}
Denote the dimension of $\cN$ by~$k$. Clearly, there is no loss of generality in assuming that $2\le k\le n-2$.

Let $Z\in\cS$ be such that $\cN$ is not $Z$-invariant. Since $\rk(Z_{21})=1$, by choosing appropriate bases $\{e_1,\dots,e_k\}$ for $\cN$ and $\{e_{k+1},\dots,e_n\}$ for $\cN^\perp$, we may assume that only the $(1,1)$-entry of $Z_{21}$ is non-zero. 

Consider the matrix $Z^2\in\cS$. Its $(2,1)$-block is equal to $Z_{21}Z_{11}+Z_{22}Z_{21}$. Notice that only the first row of the matrix $Z_{21}Z_{11}$ may contain non-zero entries and only the first column of the matrix $Z_{22}Z_{21}$ may contain non-zero entries. Since the rank of $Z_{21}Z_{11}+Z_{22}Z_{21}$ is assumed to be at most one, we conclude that one of the matrices $Z_{21}Z_{11}$ or $Z_{22}Z_{21}$ must satisfy the property that all its entries except, perhaps, the $(1,1)$-entry, are equal to zero. If all but the $(1,1)$-entry of $Z_{21}Z_{11}$ are zero, then $Z_{11}$ (and, hence, $Z$) leaves invariant the space $\span\{e_2,\dots,e_k\}$. If all but the $(1,1)$-entry of $Z_{22}Z_{21}$ are zero, then in the first column of $Z_{22}$ only the first entry may be non-zero, so that $Z$ leaves invariant the space $\span\{e_1,\dots,e_k,e_{k+1}\}$.
\end{proof}

Now we are ready to prove the main theorem of the paper.

\begin{theorem}\label{main-thm}
Let $\cG$ be a group of unitary $n\times n$ matrices. If $\rk(AB-BA)\le 2$ for all $A,B\in\cG$, then there is a subspace $\cM$ of $\bbC^n$ such that $1\le\dim\cM\le 3$ and $\cG\subseteq\cG_1\oplus\cG_2$ with $\cG_2$ abelian, where the direct sum is with respect to the decomposition $\bbC^n=\cM\oplus\cM^\perp$.
\end{theorem}
\begin{proof}
Clearly, there is no loss of generality in assuming that $\cG$ is not abelian and $n\ge 4$. Moreover, we may also assume that $\cG=\overline{\bbT\cG}$ where $\bbT$ is the unit circle on the complex plane.

Since $\cG=\overline{\cG}$, it is a compact Lie group, so \cite[Theorem~5]{BGM04} implies that $\cG$ contains a finite non-abelian subgroup. It follows that $\cG$ contains a minimal non-abelian subgroup. By \cite[Lemma 4.2.9]{RR00}, every minimal non-abelian finite group admits an invariant subspace $\cN$ such that the restriction of the group to $\cN$ is, after a similarity, generated by two matrices $\alpha A$ and $\beta S$ where $A$ is a non-scalar diagonal matrix, $S$ is the cyclic permutation, and $\alpha,\beta\in\bbT$. Since $\cG=\bbT\cG$, we conclude that $\cG$ contains a subgroup $\cG_0$ whose restriction to $\cN$ is equal (in an appropriate basis) to the group $\cG(p,q,A)$.

It follows from Proposition~\ref{p-q-restrictions} and Lemma~\ref{t.G0p3q2} that, without loss of generality, $p=2$. Since $\cG(2,q,A)$ is not abelian, it contains a matrix $C$ of the form $XYX^{-1}Y^{-1}$ different from the identity. By the properties of $\cG(p,q,A)$, this matrix is necessarily diagonal, and its diagonal entries are $q$-roots of the unity. Since $\det(C)=1$, we have $C=\diag(\omega,\bar\omega)$, for some $\omega\ne 1$, $\omega^q=1$.

If $Z\in\cG$ is an arbitrary matrix, then, considering the rank of $ZC-CZ$, we conclude that, with respect to the decomposition $\bbC^n=\cN\oplus\cN^\perp$, $Z$ is represented as  
$$
Z=\begin{bmatrix}
Z_{11} & Z_{12}\\
Z_{21} & Z_{22}
\end{bmatrix},
$$
where $\rk(Z_{21})\le 1$. By Lemma~\ref{shifted-inv-subspaces}, either $Z$ admits an eigenvector in~$\cN$, or the space $\span\{\cN,Z\cN\}$ has dimension~3 and is $Z$-invariant. Notice that this space contains $\cN$ as a subspace of codimension one.

First, we claim that, assuming $\cN$ is not $\cG$-invariant, $\cG$ admits a matrix without an eigenvector in~$\cN$.

Indeed, let $V\in\cG$ be such that $\cN$ is not $V$-invariant. If $V$ does not have eigenvectors in~$\cN$, we are done. Suppose that $V$ has an eigenvector in $\cN$. Write $V$ as
$$
V=\begin{bmatrix}
V_{11} & V_{12}\\
V_{21} & V_{22}
\end{bmatrix}.
$$
Let $f\in\cN$ be an eigenvector of~$V$. Clearly, $\span\{f\}=\ker(V_{21})$ and $f$ is an eigenvector for~$V_{11}$. Since $G(2,q,A)$ is irreducible (see, e.g., \cite[Lemma 4.2.8]{RR00}), there exists $U\in G(2,q,A)$ such that $f$ is not an eigenvector of $UV_{11}$. There exists a matrix $Z\in\cG$ of form $U\oplus D$, where $D$ is a unitary $(n-2)\times(n-2)$ matrix. Since $\ker(DV_{21})=\ker(V_{21})=\span\{f\}$, the matrix $ZV$ does not admit eigenvectors in~$\cN$.

Let $T\in\cG$ be a matrix without eigenvectors in~$\cN$. Since $T$ is a unitary matrix, every invariant subspace of it is reducing. By Lemma~\ref{shifted-inv-subspaces}, there exists an orthonormal basis $\{e_1,e_2\}$ of~$\cN$ and a unit vector $e_3$ in~$\cN^\perp$ such that, relative to the decomposition $\bbC^n=\span\{e_1\}\oplus\span\{e_2\}\oplus\span\{e_3\}\oplus(\cN^\perp\ominus\span\{e_3\})$, $T$ is written in the form
$$
T=\begin{bmatrix}
p & q & w & 0\\
r & s & u & 0\\
0 & t & v & 0\\
0 & 0 & 0 & U
\end{bmatrix},
$$
where $r\ne 0$, $t\ne 0$, and $U$ is an $(n-3)\times(n-3)$ unitary matrix.

Let $S\in\cG$ be arbitrary. Write, relative to the same decomposition,
$$
S=\begin{bmatrix}
a & b & * & *\\
c & d & * & *\\
e & f & * & *\\
g & h & * & *
\end{bmatrix},
$$
where $a,b,c,d,e,f$ are complex numbers, $g$ and $h$ are $(n-3)$-vectors, and the symbol $*$ stands for a number or a matrix whose value does not concern us. Multiplying $T$ by $S$, we get:
$$
TS=\begin{bmatrix}
* & * & * & *\\
ar+cs+eu & br+ds+fu & * & *\\
ct+ev & dt+fv & * & *\\
Ug & Uh & * & *
\end{bmatrix}.
$$
Recall that 
$$
\rk\Big(\!\!\begin{bmatrix}
e & f\\
g & h
\end{bmatrix}\!\!\Big)\le 1
\quad
\mbox{and}
\quad
\rk\Big(\!\!\begin{bmatrix}
ct+ev & dt+fv\\
Ug & Uh
\end{bmatrix}\!\!\Big)\le 1.
$$
Suppose that one of the vectors $g$ and $h$ is not zero, say, $g\ne 0$. Then there exists $\alpha\in\bbC$ such that $h=\alpha g$, $f=\alpha e$ and $dt+fv=\alpha(ct+ev)$. Since $t\ne 0$, we conclude that $d=\alpha c$. It follows that 
$$
\rk\Big(\!\!\begin{bmatrix}
c & d \\
e & f \\
g & h 
\end{bmatrix}\!\!\Big)=1.
$$
Repeating the same argument with the matrix $TS$ replacing the matrix $S$, we obtain
$$
\rk\Big(\!\!\begin{bmatrix}
ar+cs+eu & br+ds+fu\\
ct+ev & dt+fv\\
Ug & Uh
\end{bmatrix}\!\!\Big)=1.
$$
It follows that $br+ds+fu=\alpha(ar+cs+eu)$. Since $r\ne 0$, the only possibility is that $b=\alpha a$. However, this implies that 
$$
\rk\Big(\!\!\begin{bmatrix}
a & b\\
c & d \\
e & f \\
g & h 
\end{bmatrix}\!\!\Big)=1.
$$
This is impossible since the matrix $S$ is unitary, hence invertible.

The case $h\ne 0$ brings us to the same conclusion. Therefore $g=h=0$. Since $S$ was chosen arbitrarily, this implies that the space $\cM=\cG\cN$ is $\cG$-invariant and $\cM\perp(\cN^\perp\ominus\span\{e_3\})$. Under the assumption that $\cN$ is not $\cG$-invariant, this means that $\cM=\span\{e_1,e_2,e_3\}$, a 3-dimensional $\cG$-invariant subspace. The rest of the conclusions of the theorem follow from Proposition~\ref{t.HZ}.
\end{proof}

\begin{corollary}\label{semigroup-version}
Let $X$ be a Banach space and $\cS=\overline{\bbR^+\cS}$ be a semigroup of operators on $X$ containing a non-zero compact operator such that the minimal rank of nonzero operators in $\cS$ is at least~$4$. If $\rk(AB-BA)\le 2$ for all $A,B\in\cS$, then $\cS$ is reducible.
\end{corollary}
\begin{proof}
Suppose that $\cS$ is irreducible. It is well-known that a non-trivial ideal of an irreducible semigroup is irreducible. Thus, there is no loss of generality in assuming that $\cS$ consists of compact operators.

Denote the minimal non-zero rank of operators in $\cS$ by~$r$. By \cite[Lemma 8.1.15]{RR00}, $r$ is finite and there exists an idempotent $E\in\cS$ of rank~$r$. Let $\cS_0=E\cS E|_{\mathrm{Range}\,E}$. Then $\cS_0$ is represented as a semigroup of $r\times r$ matrices. Moreover, every member of this semigroup is either invertible or zero, by the minimality of the rank $r$ in $\cS$. Also, as a compression of an irreducible semigroup, the semigroup $\cS_0$ must be irreducible. By \cite[Lemma 3.1.6]{RR00}, $\cS_0\setminus\{0\}$ is a group of matrices. Moreover, there exists a group $\cG$ of unitary matrices such that, after a similarity, $\cS_0\setminus\{0\}\subseteq \bbR^+\cG$. Clearly, $\cG$ must be irreducible, too. Also, the proof of \cite[Lemma 3.1.6]{RR00} shows that the group $\cG$ is, in fact, similar to the group $\{\frac{1}{r(T)}\,T\mid T\in\cG_0\}$. Hence, the condition $\rk(AB-BA)\le 2$ holds for all $A,B\in\cG$. This, obviously, contradicts the conclusion of Theorem~\ref{main-thm}.
\end{proof}

We remark that the condition about the rank in Corollary~\ref{semigroup-version} cannot be improved. This is clear if the minimal rank is allowed to be equal to~$2$ (take, for example, the group of $2\times 2$ unitaries). The following proposition exhibits an example of an irreducible group of $3\times 3$ unitary matrices with the property $\rk(AB-BA)\le 2$ for all $A,B$ in the group.

\begin{proposition}\label{3x3-irreducible}
Let 
$$
T=\begin{bmatrix}
-1 & 0 & 0 \\
0 & 1 & 0 \\
0 & 0 & -1 \\
\end{bmatrix}
\quad
\mbox{and}
\quad
S=\begin{bmatrix}
0 & 0 & 1 \\
1 & 0 & 0 \\
0 & 1 & 0 \\
\end{bmatrix}.
$$
Then the group $\cG=\langle T,S\rangle$ is irreducible and satisfies the condition that $\rk(AB-BA)\le 2$ for all $A,B\in\cG$.
\end{proposition}
\begin{proof}
By \cite[Lemma 4.2.8]{RR00}, the group $\cG$ is irreducible. Let us show that $\rk(AB-BA)\le 2$ for all $A,B\in\cG$.

Observe that every member of $\cG$, being a finite product of matrices $T$, $S$, $T^{-1}$ and $S^{-1}$, can be written in one of the following three forms:
$$
\begin{bmatrix}
\alpha & 0 & 0 \\
0 & \beta & 0 \\
0 & 0 & \gamma \\
\end{bmatrix},
\quad
\begin{bmatrix}
0 & \alpha & 0 \\
0 & 0 & \beta \\
\gamma & 0 & 0 \\
\end{bmatrix},
\quad
\mbox{or}
\quad
\begin{bmatrix}
0 & 0 & \alpha \\
\beta & 0 & 0 \\
0 & \gamma & 0 \\
\end{bmatrix},
$$
with $\alpha,\beta,\gamma\in\{1,-1\}$. Moreover, among the numbers $\alpha,\beta,\gamma$ exactly two or none are equal to $-1$, the rest being equal to~$1$. For $A\in\cG$, let us refer to the particular form of $A$ among the three forms above as the \emph{pattern} of~$A$.

A routine check shows that for all matrices $A$ and $B\in\cG$, the patterns of $AB$ and $BA$ are the same. Hence, the difference $AB-BA$ must have the same pattern, too. Now, since there are exactly zero or two elements equal to $-1$ among the non-zero elements of $AB$ and $BA$, a quick check shows that there is at least one entry $(i,j)$ such that $(AB)_{ij}$ and $(BA)_{ij}$ are both equal to $1$ or to $-1$ simultaneously. But this means that the difference $AB-BA$ has at most two non-zero entries, so that $\rk(AB-BA)\le 2$.
\end{proof}

\section{On the structure of the group $\mathcal G(p,q,A)$}\label{s.GpqA} 

For prime numbers $p$ and $q$, let $\mathcal G=\mathcal G(p,q,A)$ be the irreducible group with generators $A$ and $S$ as defined before. These groups played a central role in our arguments from Section~\ref{main-section}. In the present section, we will further study the structure of these groups in terms of the following parameters:
\begin{eqnarray}
\rho&=&\min\{{\rm rank}(D-I)\neq 0: D\in\mathcal G; D~{\rm diagonal}\}\\
r&=&\max\{{\rm rank}(XYX^{-1}Y^{-1}-I):X,Y\in \mathcal G\}
\end{eqnarray}
\noindent {\bf Note.} Clearly, $1\leq\rho\leq r\leq p$.

Throughout the remainder of the paper, $\mathcal G=\mathcal G(p,q,A)$ for some $p,q, A$.  If $p,q$ are fixed, we may also write $\mathcal G_A$, $\rho_A$ and $r_A$ to denote $\mathcal G(p,q,A)$, $\rho$ and $r$, respectively.

\begin{theorem}\label{t.DsubsetC} Let $\mathcal D_A$ be the collection of all diagonal matrices in $\mathcal G_A=\mathcal G(p,q,A)$ and let  $\mathfrak{S}$ be the subgroup generated by $S$. Also, let $\mathcal C_A$ be the commutator subgroup of $\mathcal G_A$. Then
$\mathcal G_A =\mathcal D_A \mathfrak{S} =\mathfrak{S} \mathcal D_A$ and
 $\mathcal C_A\subset \mathcal D_A$. Moreover,  if $~\mathcal C_A\neq \mathcal D_A$, then one of the following cases holds.
\begin{enumerate}
\item $\mathcal C_A$ contains no nonscalar matrix. Then $p/2\leq \rho_A\leq r_A=p=q$ and $ \mathcal C_A=\{\eta I:~\eta^p=1\}$.
\item $\mathcal C_A$ contains nonscalar matrices and for any nonscalar  $B\in \mathcal C_A$,  $\mathcal C_B=\mathcal D_B$, $2\leq\rho_B\leq r_B\leq r_A$ and $\rho_A\leq \rho_B \leq 2\rho_A$.
\end{enumerate} 
\end{theorem}
\begin{proof} For convenience, we drop the subscript $A$ and will only maintain the subscript $B$ to avoid confusion.
Consider the general word
\begin{equation}\label{e.words}
G=A^{\alpha_1}S^{\beta_1}A^{\alpha_2}S^{\beta_2}\cdots
A^{\alpha_m}S^{\beta_m}\in \mathcal G
\end{equation}
for some integers $m,\alpha_1, \beta_1, \cdots, \alpha_m,
\beta_m$. Since
\begin{equation}
\label{e.diagonal}S^{\beta}
A^{\alpha} S^{-\beta}\in \mathcal D,~\forall \alpha,\beta\in\mathbb{Z},
\end{equation}
it follows that every word of the form (\ref{e.words}) can be
rewritten as
\begin{equation}\label{e.simplifiedwords}
G=DS^\gamma,~{\rm for~some}~ D\in \mathcal D,
\end{equation}
where $\gamma=\beta_1+\beta_2+\cdots+\beta_m$. Now,  $G$ is
diagonal if and only if $\gamma=0~({\rm mod}~p)$. Then
$\mathcal G=\mathcal D\mathfrak S$ and $\mathcal
C\subset\mathcal D$. Since $\mathcal G^{-1}=\mathcal G$, it follows that $\mathcal G=\mathfrak S\mathcal D$.

To prove $(i)$, assume $\mathcal C$ contains no nonscalar matrix. Since  $\mathcal C\neq\{I\}$, there exists $ C=\eta  I$ for some complex number $\eta\neq 1$ and some $C\in\mathcal C$.  It is easy to see that $\omega^q=1$. Also, $\omega^p=\det(C)=1$. Hence,  $q|p$ and thus  $q=p=r$. Since  $\mathcal C$ is a group, $\mathcal C=\{\eta I:~\eta^p=1\}$. Now, if rank$(D-I)=\rho <p/2$, then $D$ and $SD^{-1}S^{-1}$ each have at most  $\rho$ entries different from $1$ and, hence, $DSD^{-1}S^{-1}\neq\eta I$ for some $\eta\in\mathbb C$; a contradiction.

For $(ii)$, assume there exists a nonscalar
$B\in\mathcal C$. Then the subgroup $\mathcal G_B$ of $\mathcal G$  is
nonabelian and  the relations  
(\ref{e.diagonal}) and (\ref{e.simplifiedwords}) can be
sharpened as follows:
\begin{equation}\label{e.diagonal2}S^{\beta} B^{\alpha} S^{-\beta}=
B^{\alpha}B^{-\alpha}S^{\beta}B^{\alpha}~
S^{-\beta}\in \mathcal C_B,~\forall \alpha,\beta\in\mathbb{Z},
\end{equation}
\begin{equation}\label{e.simplifiedwords2}
G=DS^\gamma,~{\rm for~some}~ D\in \mathcal C_B.
\end{equation}
This shows that $\mathcal D_B\subset \mathcal C_B$ which proves $\mathcal D_B= \mathcal C_B$. Since $\det(C)=1$ for all $C\in\mathcal C$, it follows that rank$(D-I)\geq 2$ whenever $I\neq D\in\mathcal D$. The inequality $\rho_B\leq 2\rho_A$ follows from the fact that if rank$(D-I)=\rho$, then rank$(DSD^{-1}S^{-1}-I)\leq 2\rho$ and the rest of $(ii)$ is clear.
\end{proof}
  The next corollary studies the case $\rho=1$. We continue to use the notation established in the previous paragraphs.
\begin{corollary}\label{c.rho1} It is always true that  $2\leq r\leq p$ and, if  $I\neq C\in\mathcal C$, then ${\rm rank}(C-I)\geq 2$. In particular, if $\rho=1$, then one of the following cases holds.
\begin{enumerate}
\item    $r=p=q=2$. In this case, $\mathcal C=\{I,-I\}\subset\mathcal D=\{I,-I,{\rm diag}(1,-1),{\rm diag}(-1,1)\}$.
\item $\mathcal C$ contains nonscalar matrices and for any nonscalar $B\in \mathcal C$, $\rho_B\geq 2$ and $\mathcal  C_B=\mathcal D_B$. The lower bound $2$ is attained for some $B$.
\end{enumerate}
 As a partial converse, if $p=q=2$, then $\rho=1$,  $\mathcal D =\{I,- I,{\rm diag}(1,-1),-{\rm diag}(1,-1)\}$ and $\mathcal C=\{I,- I\}$.
\end{corollary}
\begin{proof} Observe that if ${\rm rank}(X^{-1}Y^{-1}XY-I)=1$, then $1\neq\det(X^{-1}Y^{-1}XY)= 1$; a contradiction. Thus, $2\leq r\leq p$. Now, if $D\in\mathcal D$ and ${\rm rank}(D-I)=1$, then ${\rm det}(D)\neq 1$ and, hence, $D\notin\mathcal C$. In particular, if $\rho=1$, then $\mathcal D\neq\mathcal C$ and, in view of Theorem \ref{t.DsubsetC}, one of the following cases holds.

{\bf Case 1.} $p/2\leq 1 \leq r= p=q$ which implies that $r=p=q=2$ and $\mathcal D\neq \mathcal C=\{I,-I\}$. Thus $\mathcal D=\{I,- I,{\rm diag}(1,-1),{\rm diag}(-1,1)\}$ is the only choice left.

{\bf Case 2.} There exists a nonscalar $B\in\mathcal C$ and for any such $B$, $\mathcal C_B=\mathcal D_B$ and $\rho_B\geq 2$. Now, if $D\in \mathcal D$ has exactly one diagonal entry different from $1$, then $DSD^{-1}S^{-1}$ is a commutator with exactly two diagonal entries different from $1$.

 Conversely, if $p=q=2$, then ${\rm rank}(C-I)=2$ whenever $I\neq C\in \mathcal C$, which implies that $\mathcal D\neq \{I,-I \} = \mathcal C$. Thus, $\mathcal D=\mathcal C\cup\{{\rm diag}(1,-1),{\rm diag}(-1,1)\}$  and, hence, $\rho=1$. 
\end{proof}

The following theorem studies the case $\rho=2$.
\begin{theorem}\label{t.rho=2} If $\rho=2$, then either
\begin{enumerate}
\item[(i)] $r=p$ and $q>2$,  or
\item[(ii)] $r=p-1$, $q=2$.
\end{enumerate}
\end{theorem}
\begin{proof} If $p=2$, then $r=2$. Also, $q>2=p$ by Corollary~\ref{c.rho1}.

So, we assume $p\geq 3$. Let $\mathcal D_2$ be the (nonempty) collection of
all matrices $D\in\mathcal D$ such that exactly $p-2$ entries on the main
diagonal of $D$ are  equal to $1$. We claim there exists $\Delta\in \mathcal D_2$ for which exactly the first two diagonal entries are different from $1$. Let $s$ be the minimal positive integer for which there exist a positive integer $h$ and a matrix $D={\rm diag}(\lambda_1,\lambda_2,\cdots,\lambda_p)\in \mathcal D_2$ such that $\lambda_h\neq 1$ and $\lambda_{h+s}\neq 1$. Examining  $S^{-h+1}DS^{h-1}$ and $S^{-h-s+1}DS^{h+s-1}$ reveals that $1\leq s<p/2$ and allows us to assume without loss of generality that $h=1$.  Let $p-1=ms+t$ for some nonnegative integers $m,t$ with $0\leq t\leq s-1$ and, in fact, since $p$ is an odd prime, it follows that either $s=1$ or $0\leq t\leq s-2$.  Let $\lambda_1 = \omega$ and $\lambda_{s+1}=\omega^a\neq 1$ for some primitive $q^{\rm th}$ root $\omega$ of $1$ and some positive integer $a<q$. For $1\leq k\leq m-1$, assume  $\Delta_1,\Delta_2,\cdots,\Delta_k\in\mathcal D_2$ are constructed such that $\Delta_1=D$ and the first and the $(ks+1)^{\rm th}$ diagonal entries of $\Delta_k$ are $\omega^{\epsilon_k}$ and $\omega^{a^k}$, respectively, where $\epsilon_k:=(-1)^{k+1}$.  Define $\Delta_{k+1} =S^{ks} D^{a^k}S^{-ks}\Delta_k^{-1}$. This finite induction yields $\Delta_m\in\mathcal D_2$ whose first diagonal entry is $\omega^{\epsilon_m}$ and whose $(ms+1)^{\rm th}$ diagonal entry is $\omega^{a^m}$ (necessarily, $\neq1$). Now, observe that the first and the $(t+2)^{\rm nd}$ diagonal entries of $S^{t+1}\Delta_mS^{-t-1}\in\mathcal D_2$ are  $\omega^{a^m}$ and $\omega^{\epsilon_m}$, respectively. Since all other entries are equal to $1$  and  $\omega^{\epsilon_m}\neq 1$, it follows that  $\omega^{a^m}\neq 1$. By minimality, $t+2\geq s+1$; hence, $s=1$ and $t=0$. 

Thus, there exists $k\in\{1,2,\cdots,q-1\}$ such that
\begin{equation}\label{e.Delta}
\Delta={\rm diag}(\omega,\omega^k,1,1,\cdots,1)\in\mathcal D_2.
\end{equation}
Let $\Omega:=\Gamma S\Gamma^{-1}S^{-1}\in\cC$, where
\begin{equation}\label{e.Gamma}
\Gamma={\rm diag}(\omega,\omega^k,\omega,\omega^k,\cdots,\omega,\omega^k,1)=\Pi_{j=0}^{(p-3)/2}S^{2j}\Delta S^{-2j}\in\mathcal D.
\end{equation}  Hence
\begin{equation}
\Omega={\rm diag}(\omega,\omega^{k-1},\omega^{1-k},\omega^{k-1},\cdots,\omega^{1-k},\omega^{k-1},\omega^{-k}).
\end{equation}

Let us assume $q\geq 3$ and settle the problem in this case. We claim $k\geq 2$; otherwise, 
$$
S^{-1}\Delta S\Pi_{i=1}^{p-2}S^i\Delta^{(-1)^i}S^{-i}={\rm diag}(\omega^2,1,1,\cdots,1)\in\mathcal D
$$
and rank$(D-I)=1$; a contradiction.  Therefore, $k\geq 2$ and the proof of part $(i)$ follows from the fact that $r={\rm rank}(\Omega-I)=p$.

All we have to do now is settle the case  $p>q=2$. In (\ref{e.Gamma}), $\omega=\omega^k=-1$ and one can deduce that
\begin{equation}
\Delta':=\Delta S\Delta S^{-1}={\rm diag}(-1,1,-1,1,1,\cdots,1)\in\mathcal D_2.
\end{equation}
Choose a positive integer $u$ such that $p=4u\pm 1$. Define  $\Omega' :=\Gamma' S(\Gamma')^{-1} S^{-1}\in\cC$, where 
\begin{equation}\label{e.Gamma'}
\Gamma' ={\rm diag}(-1,1,-1,1,\cdots,-1,1,-1)=\Pi_{j=0}^{u-1}S^{4j}\Delta'S^{-4j}\in\mathcal D.
\end{equation}
Hence,
\begin{equation}
\Omega'={\rm diag}(1,-1,-1,-1,\cdots,-1,-1,-1).
\end{equation}
Since $r\geq{\rm rank}(\Omega'-I)=p-1$, it follows that $p-1\leq r\leq p$. Also, since ${\rm det}(C)=1$ for all $C\in\mathcal C$, it follows that ${\rm rank}(C)\neq p$ and we are done.
\end{proof}
Based on Theorem \ref{t.rho=2}, we can sharpen Corollary \ref{c.rho1} as follows.
\begin{corollary}\label{c.rho1sharp}
If $\rho=1$, then one of the following cases holds.
\begin{enumerate}
\item    $r=p=q=2$. In this case, $\mathcal C=\{I,-I\}\subset\mathcal D=\{I,-I,{\rm diag}(1,-1),{\rm diag}(-1,1)\}$.
\item $p=r$ and $q>2$.
\item $r=p-1$ and $q=2$.
\end{enumerate}
\end{corollary}
\begin{proof}
Part $(i)$ is the same as Part $(i)$ of  Corollary \ref{c.rho1}. Let $B\in\mathcal C$ be as in Part $(ii)$ of  Corollary \ref{c.rho1} such that $\rho_B=2$. By Theorem \ref{t.rho=2}, we have one of the following cases.

{\bf Case 1.} $r_B=p$ and $q>2$. Then $p\leq r\leq p$ which proves $(ii)$.

{\bf Case 2.} $r_B=p-1$ and $q=2$. Then $r_B$ is even and, hence, $p$ is odd.
If $r$ were equal to $p$, we would have $-I\in\cC$ which is impossible since the determinant of every member of $\cC$ is equal to one. This proves~$(iii)$.
\end{proof}

The following corollary studies the case $r=2$; its easy proof is left to the interested reader.
\begin{corollary}\label{c.requal2}
If $r=2$, then  one of the following cases holds.
\begin{enumerate}
\item[(i)]  $\rho=1$ and $p=q=2$. In this case,  $$\mathcal C=\{I,-I\}\subset\mathcal D=\{I,-I,{\rm diag}(1,-1),{\rm diag}(-1,1)\}.$$
\item[(ii)] $\rho=1$, $p=2$ and $q>2$. In this case,  $$\mathcal C=\{{\rm diag}(\omega,\bar{\omega}):\omega^q=1\}\subset \mathcal D=\{{\rm diag}(\omega,\eta):\omega^q=\eta^q=1\}.$$
\item[(iii)] $\rho=1$, $p=3$ and $q=2$. In this case,
\begin{eqnarray}\mathcal C&=&\{I,{\rm diag}(1,-1,-1),{\rm diag}(-1,1,-1),{\rm diag}(-1,-1,1)\},\\
\mathcal D&=&\mathcal C\cup\{-I,{\rm diag}(-1,1,1), {\rm diag}(1,-1,1),{\rm diag}(1,1,-1)\}.
\end{eqnarray}
\item $\rho=2$, $p=2$ and $q>2$. In this case, $\mathcal C=\mathcal D=\{{\rm diag}(\omega,\eta):\omega^q=\eta^q=1\}.$
\item $\rho=2$, $p=3$ and $q=2$. In this case, $$\mathcal C=\mathcal D=\{{\rm diag}(\omega,\bar{\omega}):\omega^q=1\}.$$
\end{enumerate}
\end{corollary}


\end{document}